\documentclass{amsart}

\title{Motivic classes of commuting varieties via power structures}
\author{Jim Bryan and Andrew Morrison}
\date{\today}
\address{
Department of Mathematics\\
University of British Columbia \\
Room 121, 1984 Mathematics Road  \\
Vancouver, B.C., Canada V6T 1Z2  
}

\usepackage{amsmath}
\usepackage{amsmath,amsthm,amsfonts}
\usepackage{times}
\usepackage[all]{xy}
\newtheorem{theorem}{Theorem}
\newenvironment{primedtheorem}[1]
  {\renewcommand{\thetheorem}{\ref{#1}$'$}%
   \addtocounter{theorem}{-1}%
   \begin{theorem}}
  {\end{theorem}}

\newtheorem{lemma}[theorem]{Lemma}

\theoremstyle{definition}

\newtheorem{def-theorem}[theorem]{Definition-Theorem}
\newtheorem{remark}[theorem]{Remark}
\newtheorem{defn}[theorem]{Definition}

\newcommand{\CC} {\mathbb{C}}          
\newcommand{\NN} {\mathbb{N}}		
\newcommand{\ZZ} {\mathbb{Z}}		

\renewcommand{\AA} {\mathbb{A}}
\newcommand{\LL} {\mathbb{L}}
\newcommand{\FF} {\mathbb{F}}

\newcommand{\Hom}{\operatorname{Hom}}

\newcommand{\End}{\operatorname{End}}
\newcommand{\GL}{\operatorname{GL}}

\newcommand{\nvec}{\underline{\mathbf{n}}}
\newcommand{\tvec}{\underline{\mathbf{t}}}
\newcommand{\kvec}{\underline{\mathbf{k}}}
\newcommand{\bvec}{\underline{\mathbf{b}}}
\newcommand{\ivec}{\underline{\mathbf{i}}}
\newcommand{\zerovec}{\underline{\mathbf{0}}}

\newcommand{\Sym}{\operatorname{Sym}}
\newcommand{\Coh}{\operatorname{Coh}}
\newcommand{\Cr}{\sqrt[r]{\CC }\times \CC }

\begin{document}

\begin{abstract}
We prove a formula, originally due to Feit and Fine, for the class of
the commuting variety in the Grothendieck group of varieties. Our
method, which uses a power structure on the Grothendieck group of
stacks, allows us to prove several refinements and generalizations of
the Feit-Fine formula. Our main application is to motivic
Donaldson-Thomas theory.
\end{abstract}

\maketitle 



\section{Introduction}

Let $K_{0} (Var_{\CC })$ be the Grothendieck group of varieties
over $\CC $, i.e. the free Abelian group generated by isomorphism
classes of varieties over $\CC $ with the relation
\[
[V] = [V-Z]+[Z]
\]
for closed subvarieties $Z\subset V$. Cartesian product induces a
ring structure on $K_{0} (Var_{\CC })$. We refer to the class $[V]$
of a variety $V$ as the \emph{motivic class} of $V$.

Let $C (n)$ be the variety of commuting $n$ by $n$ matrices:
\[
C (n) = \{\,(A,B)\in \End (n)^{2}\,\, :\,\, [A,B]=0 \}.
\]

The motivic class of $C (n)$ is given by a formula which is
essentially due to Feit and Fine \cite{Feit-Fine}\footnote{Feit and Fine were counting points in $C (n)$ over the finite field $\FF _{q}$, but their method can be made to work in the motivic setting.}:
\begin{equation}\label{eqn: Feit and Fine formula for C (n)}
[C (n)] = [\GL (n)]\, \sum _{\alpha \vdash n}\, \prod _{k=1}^{\infty
}\, \frac{[\End (b_{k} (\alpha ))]}{[\GL (b_{k} (\alpha ))]}\, [\AA
_{\CC }^{b_{k} (\alpha )}]
\end{equation}
where the sum is over partitions of $n$ with the notation that $b_{k}
(\alpha )$ is the number of parts of size $k$ in $\alpha $. The
motivic class of $C (n)$ lies in the subring given by polynomials in
the Lefschetz motive $\LL :=[\AA _{\CC }^{1}]$. This follows easily
from equation~\eqref{eqn: Feit and Fine formula for C (n)} using the
elementary formula
\[
[\GL (b)] = (\LL ^{b}-1) (\LL ^{b}-\LL )\dots (\LL ^{b}-\LL ^{b-1}).
\]

In this paper, we give a new proof of Equation~\eqref{eqn: Feit and
Fine formula for C (n)} using power structures, and we prove several
refinements and generalizations.

For example, we prove that each summand in Equation~\eqref{eqn: Feit
and Fine formula for C (n)} has a geometric interpretation. Let
$\alpha $ be a partition of $n$. We say that $A\in \End (n)$ has
\emph{Jordan type $\alpha $} if the Jordan normal form of $A$ has
$b_{k} (\alpha )$ blocks of size $k$ for all $k$. We define
\[
C (\alpha ) = \{\,(A,B)\in \End (n)^{2}\,\,:\,\,[A,B]=0,\text{ $A$ has Jordan type $\alpha $} \}.
\]

\begin{theorem}\label{thm: formula for C(alpha)}
The motivic class of $C (\alpha )$ in $K_{0} (Var_{\CC })$ is given by
\begin{equation}\label{eqn: formula for C(alpha)}
[C (\alpha )] =  [\GL (n)]\,  \prod _{k=1}^{\infty
}\, \frac{[\End (b_{k} (\alpha ))]}{[\GL (b_{k} (\alpha ))]}\, \LL 
^{b_{k} (\alpha )}
\end{equation}
\end{theorem}

This theorem can be concisely expressed in terms of a generating
function. Let $M_{\CC } $ the the Grothendieck group of varieties
localized at the classes of the general linear groups, namely
\[
M_{\CC } = K_{0} (Var_{\CC })[\LL ^{-1}, (1-\LL ^{-b})^{-1}, b>0].
\]

We can regard $M_{\CC }$ as a subring of $K_{0} (Var_{\CC })[[\LL
^{-1}]]$ via\footnote{By $K_{0} (Var_{\CC })[[\LL ^{-1}]]$ we really
mean the completion of $K_{0} (Var_{\CC })[\LL ^{-1}]$ in the
dimension filtration (see for example
\cite[\S~2]{Behrend-Dhillon}). Laurent series in $\LL ^{-1}$ make
sense in this ring.} the geometric series expansion of $(1-\LL
^{-b})^{-1}$.

\begin{primedtheorem}{thm: formula for C(alpha)}\label{thm: generating fnc formulas for commuting var}
The following equation holds in $M_{\CC }[[t_{1},t_{2},\dots ]]$:
\begin{equation}\label{eqn: gen fnc for C(alpha)}
\sum _{\alpha } \frac{[C (\alpha )]}{[\GL (|\alpha |)]}\, \prod
_{k=1}^{\infty } t_{k}^{b_{k} (\alpha )} = \prod _{k=1}^{\infty }\prod
_{m=1}^{\infty } \left(1-\LL ^{2-m}t_{k} \right) ^{-1}.
\end{equation}
In particular, by setting $t_{k}=t^{k}$, we get
\[
\sum _{n=0}^{\infty }\frac{[C (n)]}{[\GL (n)]}\, t^{n }=\prod _{k=1}^{\infty }\prod
_{m=1}^{\infty } \left(1-\LL ^{2-m}t^{k} \right) ^{-1}.
\]
\end{primedtheorem}
Theorems \ref{thm: formula for C(alpha)} and \ref{thm: generating fnc
formulas for commuting var} are equivalent: Equation~\eqref{eqn: gen
fnc for C(alpha)} is obtained from Equation~\eqref{eqn: formula for
C(alpha)} by multiplying by $\prod _{k}t_{k}^{b_{k} (\alpha )}/[\GL
(|\alpha |)]$, summing over partitions, reversing the order of the sum
and the product, and then applying Euler's formula
(Equation~\eqref{eqn: Euler's formula}).

Our other generalization concerns a version of the commuting variety
where the matrices are required to have a certain block form. Let 
\[
V=V_{1}\oplus \dots \oplus V_{r}
\]
where $V_{k}$ is a complex vector space of dimension $n_{k}$. Suppose
that $A,B\in \End (V)$ are of the form
\begin{align*}
A=A_{1}\oplus \dots \oplus A_{r}&&A_{k}&\in \Hom (V_{k},V_{k+1})\\
B=B_{1}\oplus \dots \oplus B_{r}&&B_{k}&\in \Hom (V_{k},V_{k})
\end{align*}
where we regard the indices as taking values in $\ZZ /r$. We say that
$A$ is of \emph{cyclic block type} and we say that $B$ is of
\emph{diagonal block type}. Matrices of this type determine a
representation of the quiver given in figure \ref{fig: diagram for C/Zr x C quiver}.

\begin{figure}[htbp]
\centering
\[ \xymatrix{ & & & \bullet \ar@(l,ur)[llld]_{A_r} \ar@(ur,ul)_{B_r}& & & \\ \bullet \ar[r]_{A_1} \ar@(dl,dr)_{B_1} & \bullet \ar[r]_{A_2} \ar@(dl,dr)_{B_2}& \bullet \ar[r]_{A_3} \ar@(dl,dr)_{B_3}& \bullet \ar@{--}[r] \ar@(dl,dr)_{B_4}& \bullet \ar[r]_{A_{r-3}} \ar@(dl,dr)_{B_{r-3}}& \bullet \ar[r]_{A_{r-2}} \ar@(dl,dr)_{B_{r-2}} & \bullet \ar@(lu,r)[lllu]_{A_{r-1}} \ar@(dl,dr)_{B_{r-1}} } \] 
\caption{Quiver associated to the matrices from Theorem~\ref{thm: formula for C (n1,...,nr)}.} \label{fig: diagram for C/Zr x C quiver}
\end{figure}
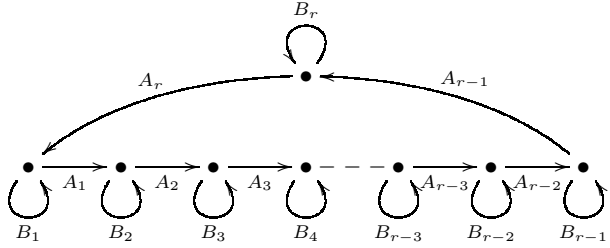

Let $\nvec= (n_{1},\dotsc ,n_{r})$ and consider the variety of
commuting pairs of endomorphisms in the above form:
\[
C (\nvec ) = \{(A,B)\in \End (V)^{2} \text{ in the above form and $[A,B]=0$} \}.
\]
\begin{theorem}\label{thm: formula for C (n1,...,nr)}
Let $C (\nvec )$ be as above and let $G (\nvec ) = \GL (V_{1})\times \dots \times \GL (V_{r})$. Then
\[
\sum _{\nvec } \frac{[C (\nvec )]}{[G(\nvec )]} \, \,t_{1}^{n_{1}}\cdots
t_{r}^{n_{r}} = \prod _{m=1}^{\infty } (1-\LL t^{m})^{-1}\,\,\prod
_{k=0}^{\infty }\,\prod _{(a,b)} \left(1-\LL
^{-k}\,t_{[a,b]}\,t^{m-1} \right)^{-1}
\]
where $t=t_{1}\cdots t_{r}$, $t_{[a,b]}=t_{a}t_{a+1}\dotsb t_{b}$, and
the last product runs over all $(a,b)\in (\ZZ /r)^{2}$.
\end{theorem}
\begin{remark}\label{rem: our thms are used in motivic DT theory}
Theorems~\ref{thm: generating fnc formulas for commuting var} and
\ref{thm: formula for C (n1,...,nr)} have been used to compute motivic
Donaldson-Thomas invariants. The Feit-Fine formula was used in
\cite{Behrend-Bryan-Szendroi} to compute the motivic DT invariants of
$\CC ^{3}$ while Theorem~\ref{thm: formula for C (n1,...,nr)} was used
in conjunction with the dimensional reduction technique of
\cite{Morrison-dimensional-reduction} to compute the motivic DT
invariants of the orbifold $\CC ^{3}/ (\ZZ /r)$ (see
\cite[\S~8]{Morrison-dimensional-reduction}).
\end{remark}

\section{Power structures on the Grothendieck groups of varieties and
stacks}\label{sec: power structures on Groth gps}

The key technical tool we use are the power structures on the
Grothendieck groups of varieties and stacks. The notion of a power
structure is due to Gusein-Zade, Luengo, and Melle-Hernandez
\cite{GZ-L-MH-power}. It is closely related to pre-lambda ring
structures, although we will not use that language here. In this
section we describe the power structures on the Grothendieck groups
and we explain a geometric interpretation which is of key importance
to us.

\begin{defn}[\cite{GZ-L-MH-power}]\label{defn: power structure}
Let $R$ be a commutative ring with identity. A \emph{power structure}
is a map $(1+tR[[t]])\times R \to 1+tR[[t]]$, denoted $(A
(t),M)\mapsto A (t)^{M}$ taking a pair $(A (t),M)$ consisting of $A
(t)$, a formal power series in $t$ with coefficients in $R$ having
constant term 1, and $M\in R$ to a formal power series $A (t)^{M}$
satisfying the following expected properties of exponentiation:
\begin{enumerate}
\item $A (t)^{0} = 1$
\item $A (t)^{1} = A (t)$
\item $(A (t)\cdot B (t))^{M} = A (t)^{M}\cdot B (t)^{M}$
\item $A (t)^{M+N}=A (t)^{M}\cdot A (t)^{N}$
\item $A (t)^{M\cdot N}= (A (t)^{M})^{N}$
\item $(1+t)^{M} = 1+Mt+O (t^{2})$
\item $A (t^{k})^{M} = (A (t))^{M}|_{t\mapsto t^{k}}$.
\end{enumerate}
\end{defn}

The Grothendieck group of varieties $K_{0} (Var_{\CC })$ has a natural
power structure uniquely determined by the equation
\[
\left(\sum _{k=0}^{\infty }t^{k} \right)^{[X]} = \sum _{n=0}^{\infty } [\Sym ^{n}X] t^{n}
\]
where $X$ is a variety and $\Sym ^{n}X$ is the $n$th symmetric product
of $X$. This power structure is \emph{effective} the sense that it
respects the semi-ring $S_{0} (Var_{\CC })\subset K_{0} (Var_{\CC })$
of effective classes, i.e. those represented by the classes of
varieties. Moreover, on effective classes, $A (t)^{[M]}$ has an
important geometric interpretation due to Gusein-Zade, Luengo, and
Melle-Hernandez. 

Suppose that $M$ is a variety and 
\[
A (t)  = \sum _{k=0}^{\infty } [A_{k}]t^{k}
\]
where $A_{k}$ are varieties with $A_{0}$ a point. Then the coefficient
of $t^{n} $ in $A (t)^{[M]}$ is given by the class of the variety
\begin{equation}\label{eqn: geometric formula for nth coef of A(t)^M}
B_{n} = \bigsqcup _{\alpha \vdash n} \left(\prod _{i=1}^{\infty } M^{b_{i} (\alpha )}\setminus  \Delta  \right)\times _{S_{\alpha }} \left(\prod _{i=1}^{\infty } A_{i}^{b_{i} (\alpha )} \right)
\end{equation}
where $b_{i} (\alpha )$ is the number of parts of size $i $ in the
partition $\alpha $, $\Delta $ is the large diagonal in $\prod
_{i=1}^{\infty }M^{b_{i} (\alpha )} $, and $S_{\alpha }$ is the
product $\prod _{i=1}^{\infty }S_{b_{i} (\alpha )}$ of symmetric
groups acting on each factor in the obvious way. The variety $B_{n}$
has the following geometric interpretation; it parameterizes maps
\begin{equation}\label{eqn: interpretation of coefs of A(t)^M as charged particles on M}
\phi :S\to \bigsqcup _{i=1}^{\infty } A_{i}
\end{equation}
where $S\subset M$ is a finite subset and $\phi $ satisfies
\[
n= \sum _{x\in S} i (\phi (x))
\]
where $i:\sqcup_{i}A_{i}\to \NN $ is the tautological map taking the
value $i$ on the component $A_{i}$. We think of $B_{n}$ as
parameterizing a finite set of ``particles'' on $M$ of total
``charge'' $n$ where the ``internal state space'' of a charge $i$ particle
is $A_{i}$.

The power structure on $K_{0} (Var_{\CC })$ and the above geometric
interpretation have an extension to $K_{0} (Stck_{\CC })$, the
Grothendieck group of stacks having affine stabilizer groups:

It is shown in \cite{Ekedahl1} that $K_{0} (Stck_{\CC })$ is isomorphic
to $M_{\CC } = K_{0} (Var_{\CC })[\LL ^{-1},(1-\LL ^{n})^{-1}]$, the
Grothendieck group of varieties localized at the classes $\LL $ and
$(1-\LL ^{n})$ for $n\geq 1$ (this is equivalent to localizing at the
classes $[GL_{n}]$ for all $n$). Via the expansion of the classes
$(1-\LL ^{n})^{-1}$ as Laurent series in $\LL ^{-1}$, there is a map
\[
K_{0} (Stck_{\CC })\to K_{0} (Var_{\CC })[[\LL ^{-1}]]
\]
to the completion of $K_{0} (Var_{\CC })[\LL ^{-1}]$ in the dimension
filtration (which we denote somewhat abusively by $K_{0} (Var_{\CC
})[[\LL ^{-1}]]$).

The power structure on $K_{0} (Var_{\CC })$ has a unique extension to
$K_{0} (Stck_{\CC })$ and $K_{0} (Var_{\CC })[[\LL ^{-1}]]$
characterized by the property
\begin{equation}\label{eqn: (1-t)^(-LL^k) = (1-LL^k t)^(-1)}
(1-t)^{-\LL ^{k}[M]} = (1-\LL ^{k}t)^{-[M]}
\end{equation}
for all $k\in \ZZ $ (see \cite{Ekedahl2, GZ-L-MH-stacks}) and
varieties $M$. For $k\geq 0$, the above formula is equivalent to the
statement that $[\Sym ^{n} ( \CC ^{k}\times M)]=[\CC ^{nk}\times \Sym
^{n} (M)]$ in $K_{0} (Var_{\CC })$ which is a lemma due to Totaro
\cite[Lemma~4.4]{Gottsche-Motive}.

The power structure on $K_{0} (Stck_{\CC })$ no longer respects the
semi-ring $S_{0} (Stck_{\CC })$ of effective classes, i.e. those
spanned by stacks (see \cite[\S~3]{GZ-L-MH-stacks}). However, the
following partial effectivity result holds:
\begin{lemma}\label{lem: geometric formula for A(t)^M holds when A is effective stack and M is a variety}
Suppose that $M$ is a variety and $A_{i}$ are stacks where $A_{0}$ is
a point. Then the coefficients of $\left(\sum _{i=0}^{\infty
}[A_{i}]t^{i} \right)^{[M]}$ are the classes of stacks, in particular
they are given by equation \eqref{eqn: geometric formula for nth coef
of A(t)^M} and the geometric interpretation given by
equation~\eqref{eqn: interpretation of coefs of A(t)^M as charged
particles on M} continues to hold.
\end{lemma}
\begin{proof}
Let $B (t)=\sum _{n=0}^{\infty }[B_{n}]t^{n}$ where $B_{n}$ is the
stack defined by equation~\eqref{eqn: interpretation of coefs of
A(t)^M as charged particles on M}. We need to show that $A (t)^{[M]}=B
(t)$.

It is a formal consequence of the existence of the power structure and
equation~\eqref{eqn: (1-t)^(-LL^k) = (1-LL^k t)^(-1)} that
exponentiation commutes with the substitution $t\mapsto \LL ^{N}t$,
i.e.
\[
C (\LL ^{N}t)^{M} = \left(C (t)^{M} \right)|_{t\mapsto \LL ^{N}t}
\]
for any series $C (t)$ beginning with 1 (see \cite[Statement
2]{GZ-L-MH-power}).

We fix positive integers $d$ and $D$. Using the dimension filtration
$F_{-D}\subset K_{0} (Var_{\CC })[[\LL ^{-1}]]$ we may consider the
series $A (t)$ modulo $t^{d} $ and modulo elements of dimension $\leq
-D$. Since $A_{1},\dotsc ,A_{d-1}$ are the classes of stacks, there
exists a polynomial
\[
\tilde{A} (t) = \sum _{i=0}^{d-1} \tilde{A}_{i} t^{i}
\]
and an integer $N=N (d,D)$ such that 
\[
\tilde{A}(t) = A (t) \mod (t^{d}, F_{-D})
\]
and such that $\LL ^{iN}\tilde{A}_{i}$ is the class of a variety.
Since the desired formula holds for series whose coefficients are
varieties, it holds for the series $\tilde{A} (\LL ^{N}t)$ and hence
we have that
\[
A (\LL ^{N}t)^{[M]} = B (\LL ^{N}t) \mod (t^{d}, F_{-D})
\]
and thus by the substitution rule, we have
\[
A (t)^{[M]} = B (t) \mod (t^{d}, F_{-D}).
\]
Since the equality holds for arbitrary $d$ and $D$, it must hold in
$K_{0} (Var_\CC)[[\LL ^{-1}]][[t]]$.
\end{proof}

Power structures can be easily extended to accommodate power series in
several or even an infinite number of variables
\cite{GZ-L-MH-Integration,GZ-L-MH-Hilb}. Let $\tvec = (t_{1},\dotsc
,t_{r})$ denote an $r$-tuple of variables ($r$ could be countably
infinite) and let $R[[\tvec ]]$ be the ring of formal power series
$\sum _{\kvec } A_{\kvec } \tvec ^{\kvec }$ where the sum is over
$r$-tuples $\kvec = (k_{1},\dotsc ,k_{r})$ of non-negative
integers\footnote{If $r$ is infinite, then we require all but a finite
number of the $k_{i}$s to be zero.}. We adopt the multi-index product
convention throughout: $\tvec ^{\kvec }:=\prod _{j} t_{j}^{k_{j}}$.

In \cite{GZ-L-MH-Integration}, Gusein-Zade, Luengo, and
Melle-Hernandez, extend power structures to multi-variable power
series rings and they extend the geometric interpretation given by
equations \eqref{eqn: geometric formula for nth coef of A(t)^M} and
\eqref{eqn: interpretation of coefs of A(t)^M as charged particles on
M} to the multi-variable case. Namely, if $[A_{\kvec }]$ is the class
of a variety (stack) and $[M]$ is the class of a variety, then the
coefficient of $\tvec ^{\nvec }$ in the series $\left(\sum _{\kvec }
[A_{\kvec }] \tvec ^{\kvec } \right)^{[M]}$ is the variety (stack)
parameterizing maps
\[
\phi :S\to \bigsqcup _{\kvec \neq \zerovec } A_{\kvec }
\]
where $S\subset M$ is a finite subset and 
\[
\sum _{x\in S} \ivec (\phi (x)) = \nvec 
\]
where 
\[
\ivec :\sqcup A_{\ivec}\to \ZZ _{\geq 0}^{r}
\]
is the tautological index map. This again can be interpreted as
parameterizing ``particles'' on $M$ of total ``charge'' $\nvec $ where
``charge'' now is given by an $r$-tuple of integers and the ``internal
state space'' of a charge $\kvec $ particle is parameterized by
$A_{\kvec }$. The proof of Lemma~\ref{lem: geometric formula for
A(t)^M holds when A is effective stack and M is a variety} works in
this slightly more general setting.

We note that unpacking the definition of the multivariable power structure, we find that 
\[
(1-\tvec ^{\nvec })^{-\LL ^{k}} = (1-\LL ^{k} \tvec ^{\nvec })^{-1}
\]
for each monomial $\tvec ^{\nvec }$.

\section{The main argument}\label{sec: the main argument} In this
section, we first use the power structure on $K_{0} (Stck_{\CC })$ to
provide a short and easy proof of the Feit-Fine formula
(equation~\eqref{eqn: Feit and Fine formula for C (n)}). We then
refine the basic argument to prove our refinements and generalizations
of the Feit-Fine formula.

\subsection{Proof of the Feit-Fine formula}\label{subsec: proof of Feit-Fine}
We begin with the observation that the motivic class
\[
\frac{[C (n)]}{[GL (n)]} \in K_{0} (Var_{\CC })[\LL ^{-1},(1-\LL
^{b})^{-1}] \cong K_{0} (Stck_{\CC })
\]
has a geometric interpretation as the class of a natural
stack. Namely, let
\[
\Coh_{n} (\CC ^{2})
\]
denote the stack of coherent sheaves on the affine plane which are
supported at points and of length $n$. This is equivalent to the stack
of modules $M$ over the ring $\CC [x,y]$ such that $\dim _{\CC
}M=n$. Given a $\CC [x,y]$ module $M$ and a linear isomorphism $M\cong
\CC ^{n}$, the action of $x$ and $y$ on $M$ yields a pair $(A,B)$ of
commuting $n\times n$ matrices. Dividing out by the choice of the
isomorphism, this correspondence induces a stack equivalence
\begin{equation}\label{eqn: C(n)/GL(n)=Coh_n (C^2)}
C (n)/GL (n) \cong \Coh _{n}(\CC ^{2}).
\end{equation}
Since all $GL (n)$ torsors are trivial (essentially by definition
\cite{Ekedahl1}) in $K_{0} (Stck_{\CC })$, the above yields
\[
\sum _{n=0}^{\infty } \frac{[C (n)]}{[GL (n)]} \, t^{n} = \sum _{n=0}^{\infty } \left[\Coh _{n} (\CC ^{2}) \right]\, t^{n}
\]
in $K_{0} (Stck_{\CC })[[t]]$.

Let
\[
\Coh ^{(0,0)}_{n} (\CC ^{2})
\]
be the stack of length $n$ coherent sheaves supported at
$(0,0)\in \CC ^{2}$. Then the geometric interpretation of the power
structure on $K_{0} (Stck_{\CC })$ implies that
\begin{equation}\label{eqn: power formula 1}
\sum _{n=0}^{\infty } [\Coh _{n} (\CC ^{2})] \,t^{n} = \left(\sum
_{n=0}^{\infty }\left[\Coh _{n}^{(0,0)} (\CC ^{2}) \right] t^{n}
\right) ^{[\CC ^{2}]}.
\end{equation}
Indeed, since the stack $\Coh ^{(a,b)}_{n} (\CC ^{2})$ of length $n$
sheaves supported at $(a,b)\in \CC ^{2}$ is canonically isomorphic to
$\Coh ^{(0,0)}_{n} (\CC ^{2})$, we see that as required by
equation~\eqref{eqn: interpretation of coefs of A(t)^M as charged
particles on M}, $\Coh _{n} (\CC ^{2})$ parameterizes maps
\[
\phi :S \to \bigsqcup _{i=1}^{\infty } \Coh ^{(0,0)}_{i} (\CC ^{2})
\]
where $S\subset \CC ^{2}$ is the support set of the sheaf and $\phi $
identifies the sheaf restricted to each point of its support with the
corresponding sheaf supported at $(0,0)$.

Let 
\[
\Coh ^{(0,*\neq 0)}_{n} (\CC ^{2})
\]
be the stack of length $n$ coherent sheaves whose support lies on
points of the form $(0,b)\in \CC ^{2}$ with $b\neq 0$. Then by an
argument similar to above we have
\begin{equation}\label{eqn: power formula 2}
\sum _{n=0}^{\infty }\left[\Coh ^{(0,*\neq 0)}_{n} (\CC ^{2}) \right]
t^{n} = \left(\sum _{n=0}^{\infty }\left[\Coh ^{(0,0)}_{n} (\CC ^{2}) \right] t^{n} \right)^{[\CC -\{0 \}]}.
\end{equation}
Combining equations \eqref{eqn: power formula 1} and \eqref{eqn: power
formula 2} and applying the power structure axioms we get
\[
\sum _{n=0}^{\infty }\left[\Coh _{n} (\CC ^{2}) \right] t^{n} =
\left(\sum _{n=0}^{\infty } \left[\Coh _{n}^{(0,*\neq 0)} (\CC ^{2})
\right] t^{n}\right)^{\frac{\LL ^{2}}{\LL -1} } .
\]
Note that under the equivalence given by equation \eqref{eqn:
C(n)/GL(n)=Coh_n (C^2)}, a point in the substack $\Coh ^{(0,*\neq
0)}_{n} (\CC ^{2})\subset \Coh _{n} (\CC ^{2})$ corresponds to a pair
of matrices $(A,B)$ where the eigenvalues of $A$ are all 0 and the
eigenvalues of $B$ are all non-zero. In other words, $\Coh ^{(0,*\neq
0)}_{n} (\CC ^{2})$ is equivalent to the stack quotient:
\begin{equation*}
\{(A,B)\in \End (n)^{2}:[A,B]=0 , \text{$A$ is nilpotent, $B$ is
invertible}\}/GL (n) . 
\end{equation*}
For any partition $\lambda \vdash n$, let $J_{\lambda }$ denote the
unique nilpotent matrix in Jordan normal form having Jordan type
$\lambda $. Then using the action of $GL (n)$ to put $A$ in Jordan
normal form, we get an equivalence
\[
\Coh _{n}^{(0,*\neq 0)} (\CC ^{2})\cong \bigsqcup _{\lambda }
\left\{B\in GL (n):[J_{\lambda },B] =0
\right\}/\operatorname{Stab}_{J_{\lambda }} (GL (n)) 
\]
where $\operatorname{Stab}_{J_{\lambda }} (GL (n)) \subset GL (n)$ is
the stabilizer of $J_{\lambda }$ under the conjugation action.  Since
the equation $[J_{\lambda },B]=0$ is equivalent to the equation
$B^{-1}J_{\lambda }B=J_{\lambda }$, the numerator and the denominator
of the above stack quotient are the same. Thus in $K_{0} (Stck_{\CC
})$ we simply have
\[
\left[\Coh _{n}^{(0,*\neq 0)} (\CC ^{2}) \right] = p (n)
\]
where $p (n)$ is the number of partitions of $n$. 

Putting it all together, we get
\begin{align*}
\sum _{n=0}^{\infty }\frac{[C (n)]}{[GL (n)]} t^{n} &=\left(\sum _{n=0}^{\infty } \left[\Coh ^{(0,*\neq 0)}_{n} (\CC ^{2})  \right]t^{n} \right) ^{\frac{\LL ^{2}}{\LL -1}} \\
&=\left(\sum _{n=0}^{\infty }p (n)t^{n} \right) ^{\LL +1+\LL ^{-1}+\LL ^{-2}+\dotsb }\\
&= \prod _{k=1}^{\infty }\left(\prod _{m=1}^{\infty } (1-t^{m})^{-1} \right) ^{\LL ^{2-k}}\\
&= \prod _{k=1}^{\infty }\prod _{m=1}^{\infty } (1-\LL
^{2-k}t^{m})^{-1}.
\end{align*}
Here we are working over the completed ring $K_{0} (Var_{\CC })[[\LL
^{-1}]]$ and we have used equation \eqref{eqn: (1-t)^(-LL^k) = (1-LL^k
t)^(-1)}. The above equation is equivalent by expansion to the Feit-Fine
formula and thus completes our proof.

\subsection{Proof of Theorem~\ref{thm: formula for C(alpha)}/Theorem~
\ref{thm: generating fnc formulas for commuting var}}

This proof is a straightforward generalization of the previous proof
utilizing the multi-variable generalization of power structures.

Let $\bvec = (b_{1},b_{2},\dotsc )$ be a countable tuple of
non-negative integers with a finite number of non-zero entries. Let 
\[
\alpha  = (1^{b_{1}},2^{b_{2}},\dotsc )
\]
be the corresponding partition of $n=\sum _{k}k b_{k}$, i.e. $\alpha $
is the partition having $b_{k}$ parts of size $k$.

Let
\[
\Coh _{\bvec } (\CC ^{2})
\]
denote the stack of length $n$ coherent sheaves $F$ on $\CC ^{2}$ such
that multiplication by the $x$ coordinate on $H^{0} (\CC ^{2},F)$ has
Jordan type $\alpha $.

Via the same argument leading to equation~\eqref{eqn: C(n)/GL(n)=Coh_n
(C^2)}, we have the equivalence of stacks
\[
\frac{C (\alpha )}{GL (n)} \cong \Coh _{\bvec } (\CC ^{2}).
\]
Thus we get the equation
\[
\sum _{\bvec } \frac{[C (\alpha )]}{[GL (n)]} \tvec ^{\bvec } = \sum
_{\bvec }\left[\Coh _{\bvec } (\CC ^{2}) \right] \tvec ^{\bvec }
\]
in the ring $K_{0} (Stck_{\CC })[[\tvec ]]$. 

Let $\Coh _{\bvec }^{(0,0)} $ and $\Coh _{\bvec }^{(0,*\neq 0)} $ be
the substacks of $\Coh _{\bvec } (\CC ^{2})$ parameterizing sheaves
supported at the origin and at points of the form $(0,y) $ with $y\neq
0$ respectively. The geometric interpretation of the (multi-variable)
power structure yields
\[
\sum _{\bvec } \left[\Coh _{\bvec } (\CC ^{2}) \right] \tvec ^{\bvec }
=\left(\sum _{\bvec } \left[\Coh _{\bvec }^{(0,0)} (\CC ^{2}) \right]
\tvec ^{\bvec } \right)^{\LL ^{2}}
\]
since once again we may use translation to canonically identify $\Coh
_{\bvec }^{(x_{0},y_{0})} (\CC ^{2})$ with $\Coh _{\bvec }^{(0,0)}
(\CC ^{2})$ to see that $\Coh _{\nvec } (\CC ^{2})$ parameterizes
maps
\[
\phi :S\to  \bigsqcup _{\bvec } \Coh _{\bvec }^{(0,0)} (\CC ^{2})
\]
where $S\subset \CC ^{2}$ is the support of the sheaf and $\nvec =\sum
_{x\in S}\bvec (\phi (x))$.

Arguing as in \S~\ref{subsec: proof of Feit-Fine}, we arrive at the equation
\[
\sum _{\bvec }\left[\Coh _{\bvec } (\CC ^{2}) \right] \tvec ^{\bvec }
= \left( \sum _{\bvec }\left[\Coh ^{(0,*\neq 0)}_{\bvec } (\CC ^{2})
\right] \tvec ^{\bvec } \right) ^{\frac{\LL ^{2}}{\LL -1}}.
\]
We then argue as before to find that the motivic class of stack $\Coh
_{\bvec }^{(0,*\neq 0)} (\CC ^{2})$ is just 1 so that we get
\begin{align*}
\sum _{\bvec } \left[\Coh _{\bvec } (\CC ^{2}) \right] \tvec ^{\bvec } &= \left(\sum _{\bvec } \tvec ^{\bvec } \right)^{\LL +1+\LL ^{-1}+\LL ^{-2}+\dotsb }\\
&=\prod _{k=1}^{\infty }\prod _{m=1}^{\infty } (1-t_{m})^{-\LL ^{2-k}}\\
&=\prod _{k=1}^{\infty }\prod _{m=1}^{\infty } (1-\LL ^{2-k} t_{m})^{-1}.
\end{align*}

\subsection{Proof of Euler's formula} Applying the argument of
\S~\ref{subsec: proof of Feit-Fine} to the affine line instead of the
affine plane gives us a geometric proof of a more elementary formula
due to Euler. Let
\[
\Coh _{n} (\CC )
\]
be the stack of length $n$ coherent sheaves on the affine line. It is
equivalent to the stack of $\CC [x]$ modules of dimension $n$ and
hence
\[
\Coh _{n} (\CC ) \cong \End (n)/GL (n).
\]
Similarly, the stacks of coherent sheaves supported at the origin and
away from the origin respectively are given by nilpotent and
invertible matrices up to conjugation:
\begin{align*}
\Coh ^{0}_{n} (\CC )\cong Nil (n) /GL (n),&&\Coh ^{*\neq 0}_{n} (\CC
)\cong GL (n)/GL (n).
\end{align*}
The geometric interpretation of the power structure then implies
\begin{align*}
\sum _{n=0}^{\infty }\left[ \Coh_{n} (\CC )  \right] t^{n} &= \left(\sum _{n=0}^{\infty } \left[ \Coh ^{0}_{n} (\CC ) \right] t^{n} \right) ^{\LL }\\
\sum _{n=0}^{\infty }\left[ \Coh_{n}^{*\neq 0} (\CC )  \right] t^{n}& = \left(\sum _{n=0}^{\infty } \left[ \Coh ^{0}_{n} (\CC ) \right] t^{n} \right) ^{\LL -1}
\end{align*}
Consequently we have
\begin{align}\label{eqn: Euler's formula}
\nonumber
\sum _{n=0}^{\infty } \frac{[\End (n)]}{[GL (n)]} t^{n}&= \left(\sum _{n=0}^{\infty } \left[\Coh ^{*\neq 0}_{n} (\CC ) \right]t^{n} \right) ^{\frac{\LL }{\LL -1}}\\ \nonumber
&=\left(\sum _{n=0}^{\infty }t^{n} \right)^{1+\LL ^{-1}+\LL ^{-2}+\dotsb }\\ \nonumber
&=\prod _{k=0}^{\infty } (1-t)^{-\LL ^{-k}}\\ 
&=\prod _{k=0}^{\infty } (1-\LL ^{-k}t)^{-1}.
\end{align}
Since 
\[
\frac{[\End (n)]}{[GL (n)]} = \frac{\LL ^{n^{2}}}{(\LL ^{n}-\LL
^{n-1})\dotsb (\LL ^{n}-1)} = \frac{1}{(1-\LL ^{-1})\dotsb (1-\LL
^{-n})}
\]
we can set $q=\LL ^{-1}$ to get the more familiar version of Euler's
formula
\[
\sum _{n=0}^{\infty } \frac{t^{n}}{(1-q)\dotsb (1-q^{n})} =\prod
_{k=0}^{\infty } (1-q^{k}t)^{-1}.
\]

\subsection{Proof of Theorem~\ref{thm: formula for C (n1,...,nr)} }
This proof follows a similar script to our previous proofs. The key is
to interpret the variety $C (\nvec )$ in terms of sheaves on an
orbifold quotient of $\CC ^{2}$. Namely, we consider the orbifold
given by the following stack quotient
\[
\Cr :=\CC / (\ZZ /r) \times \CC 
\]
where $k\in \ZZ /r$ acts on $\CC $ by multiplication by $\exp\left(2\pi i k/r \right)$.

A coherent sheaf on $\Cr $ may be regarded as a coherent sheaf on $\CC
^{2}$, invariant under the action of $\ZZ /r$. If $F$ is a $\ZZ
/r$-invariant sheaf on $\CC ^{2}$, then $H^{0} (\CC ^{2},F) $ is
naturally a $\ZZ /r$ representation. For each $\nvec = (n_{1},\dotsc
,n_{r})$, let
\[
\Coh _{\nvec } (\Cr )
\]
be the stack of sheaves on $\Cr $ such that the dimension of the
weight $k$ space in the $\ZZ /r$ representation $H^{0} (\CC ^{2},F)$
is $n_{k}$.

We fix a $\ZZ /r$ representation $V=V_{1}\oplus \dotsb \oplus V_{r}$
such that $V_{k}$, the weight $k$ subrepresentation, has dimension
$n_{k}$. For any object $F\in \Coh _{\nvec } (\Cr )$, we fix an
isomorphism
\[
H^{0} (\CC ^{2},F)\cong V.
\]
Note that under this identification, multiplication by $x$ on $V$
(regarded as an $\CC [x,y]$-module) takes $V_{i}$ to $V_{i+1}$ and
multiplication by $y$ takes $V_{i}$ to $V_{i}$. Dividing out by the
choice of the isomorphism, the identification induces a stack
equivalence:
\begin{equation}\label{eqn: stack equivalence C(n1,..,nr)/G = Coh(C/ZrxC)}
C (\nvec )/G (\nvec )\cong \Coh _{\nvec } (\Cr )
\end{equation}
and hence leads to the equality
\begin{equation}\label{eqn: series for C(n1,..,nr) is the series for the stack Coh}
\sum _{\nvec }\frac{[C (\nvec )]}{[G (\nvec )]} \tvec ^{\nvec } = \sum
_{\nvec } [\Coh _{\nvec } (\Cr )] \tvec ^{\nvec }
\end{equation}
in the ring $K_{0} (Stck_{\CC})[[t_{1},\dotsc ,t_{r}]]$.

To analyze the series $\sum _{\nvec } [\Coh _{\nvec } (\Cr )] \tvec
^{\nvec }$ using power structures, we must consider the stacky points
and non-stacky points of $\Cr $ separately. We use
\[
\Coh ^{(0,*)}_{\nvec } (\Cr )\text{, and } \Coh ^{(*\neq 0,*)}_{\nvec
} (\Cr )
\]
respectively to denote the substack of sheaves supported on
respectively the stacky locus $B\ZZ /r \times \CC \subset \Cr $ and
its complement $ \CC ^{*}/ (\ZZ /r)\times \CC \subset \Cr $.

We note that $\Coh _{\nvec }^{(*\neq 0,*)} (\Cr )$ is empty unless
$n_{1}=\dotsb =n_{r}$ in which case we have an equivalence
\[
\Coh _{(n,\dotsc ,n)}^{(*\neq 0,*)} (\Cr ) \cong \Coh _{n}^{(*\neq
0,*)} (\CC ^{2})
\]
induced by the equivalence $\CC ^{*}/ (\ZZ /r)\times \CC \cong \CC
^{*}\times \CC $. Letting 
\[
t=t_{1}\dotsb t_{r}
\]
and using the power structure as in \S~\ref{subsec: proof of
Feit-Fine}, we get
\begin{align}\label{eqn: generating series for non-stacky part of C/ZrxC}
\nonumber\sum _{\nvec }[\Coh _{\nvec }^{(*\neq 0,*)} (\Cr )]\tvec ^{\nvec } 
&= \sum _{n=0}^{\infty } [\Coh _{n}^{(*\neq 0,*)} (\CC ^{2})]t^{n} \\
\nonumber&= \left(\sum _{n=0}^{\infty } [\Coh _{n}^{(*\neq 0,0)} (\CC ^{2})]t^{n}  \right)^{\LL }\\
\nonumber&= \left(\sum _{n=0}^{\infty }p (n) t^{n} \right)^{\LL }\\
\nonumber&=\left(\prod _{m=1}^{\infty } (1-t^{m})^{-1} \right)^{\LL }\\
&= \prod _{m=1}^{\infty } (1-\LL t^{m})^{-1}. 
\end{align}

Using the geometric interpretation of the multi-variable version of
the power structure, we may use arguments similar to the ones the
previous proofs to get
\begin{align}\label{eqn: formula for stacky Coh(0,*) in terms of Coh(0,*neq 0) using power struct}
\nonumber \sum _{\nvec }[\Coh _{\nvec }^{(0,*)} (\Cr )] \tvec ^{\nvec } &= \left( \sum _{\nvec } [\Coh _{\nvec }^{(0,0)} (\Cr )] \tvec ^{\nvec } \right)^{\LL }\\
&=\left(\sum _{n} [\Coh _{\nvec }^{(0,*\neq 0)} (\Cr )] \tvec ^{\nvec } \right)^{\frac{\LL }{\LL -1}}. 
\end{align}

The right hand side in the above equation will be determined using the
following lemma.

\begin{lemma}[c.f. Lemma 4.8 of \cite{Morrison-Nagao}]\label{lem: class of stacky Coh(0,*neq0) is given by the number of collections lambda(a,b)}
The class $[\Coh _{\nvec }^{(0,*\neq 0)} (\Cr )]\in K_{0} (Stck_{\CC
})$ is the positive integer given by the number of collections
$\{\lambda (a,b) \}$ consisting of partitions indexed by $(a,b)\in
(\ZZ /r)^{2}$ satisfying
\begin{equation}\label{eqn: dimension vector associated to the collection lambda(a,b)}
n_{j} = \sum _{(a,b)} \left\{|\lambda (a,b)| - \sum
_{[a,b]\not \owns j}l (\lambda (a,b)) \right\}
\end{equation}
where $l (\lambda (a,b))$ denotes the length of the partition $\lambda
(a,b)$ and $[a,b]\subset \ZZ /r$ is the ``interval'' $a,a+1,\dotsc
,b$.
\end{lemma}
\begin{proof}
Recall that an endomorphism $A$ of $V_{1}\oplus \dotsb \oplus V_{r}$
is of \emph{cyclic block type} if $A=A_{1}\oplus \dotsb \oplus A_{r}$
where $A_{i}\in \Hom (V_{i},V_{i+1})$ and $B$ is of \emph{diagonal
block type} if $B=B_{1}\oplus \dotsb \oplus B_{r}$ where $B_{i}\in \Hom
(V_{i},V_{i})$. Note that $B $ is invertible and of diagonal block
type if and only if $B\in G (\nvec)$.

The stack equivalence \eqref{eqn: stack equivalence C(n1,..,nr)/G =
Coh(C/ZrxC)} then induces an equivalence between $\Coh _{\nvec
}^{(0,*\neq 0)} (\Cr )$ and the stack quotient
\vspace{.1in}
\[
\left\{A,B\in \End (V): [A,B]=0\text{, $A$ is nilpotent of
cyclic block type, }B\in G (\nvec ) \right\}/G (\nvec ).
\] 
We now analyze Jordan normal form for nilpotent endomorphisms of
cyclic block type.

We suppose that $A$ is of cyclic block type and is nilpotent.  Jordan
canonical form says there exists vectors $e_{1},\dotsc ,e_{l}\in V$
such that the collection of non-zero vectors of the form $A^{j}e_{i}$
is a basis of $V$. Moreover, without loss of generality, we may assume
that each $e_{i}$ (and hence each $A^{j}e_{i}$) lies in a single
summand of $V=V_{1}\oplus \dotsb \oplus V_{r}$. We will say that a
vector $f\in \{e_{1},\dotsc ,e_{l} \}$ \emph{starts at $a$} and
\emph{ends at $b$} if $f\in V_{a}$ and $A^{k}f\in V_{b}$ where $k$ is
the largest integer such that $A^{k}f\neq 0$. We define the length of
$f$ to be $\left\lfloor \frac{k}{r} \right\rfloor +1$.  We define a
collection of partitions $\{\lambda (a,b) \}$ by declaring that the
number of parts of size $j$ in the partition $\lambda (a,b)$ is the
number of vectors in the collection $\{e_{1},\dotsc ,e_{l} \}$ which
start at $a$, end at $b$, and have length $j$.

Note that the dimensions of $V_{j}$ can be written in terms the
partitions $\lambda (a,b)$ and they are given precisely by
equation~\eqref{eqn: dimension vector associated to the collection
lambda(a,b)}.

The collection $\{\lambda (a,b) \}$ uniquely determines a nilpotent
endomorphism of cyclic block type up to conjugation by elements in $G
(\nvec )$.

For each collection partitions $\{\lambda (a,b) \}$, let $J_{\{\lambda
(a,b) \}}$ be the matrix in the Jordan form described above. Then we
have
\begin{multline*}
\left\{A,B\in \End (V): [A,B]=0\text{, $A$ is nilpotent of cyclic
block type, }B\in G (\nvec ) \right\}/G (\nvec ) \\
\cong \bigsqcup _{\{\lambda (a,b) \}} \left\{B\in G (\nvec ):
[J_{\{\lambda (a,b) \}},B]=0
\right\}/\operatorname{Stab}_{J_{\{\lambda (a,b) \}}} (G (\nvec ))
\end{multline*}
where the union is taken over all collections $\{\lambda (a,b) \}$
satisfying equation~\eqref{eqn: dimension vector associated to the
collection lambda(a,b)} and where
\[
\operatorname{Stab}_{J_{\{\lambda (a,b)
\}}} (G (\nvec ))\subset G (\nvec )
\]
is the stabilizer of $J_{\{\lambda (a,b) \}}$ under the action of
conjugation by elements of $G (\nvec )$.

Since the equation $[J_{\{\lambda (a,b) \}},B]=0$ is equivalent to the
equation $BJ_{\{\lambda (a,b) \}}B^{-1}=J_{\{\lambda (a,b) \}}$, the
numerator and the denominator of the above stack quotient are the
same. Therefore, in the Grothendieck group, each factor in the
disjoint union contributes 1. The lemma follows.
\end{proof}

We can now use the lemma to compute. In the below, $\{\lambda (a,b)
\}$ always denotes a collection of partitions indexed by $(a,b)\in
(\ZZ /r)^{2}$ and we use the notation $]a,b[$ to denote the complement
of the ``interval'' $[a,b]=\{a,a+1,\dotsc ,b \}\subset \ZZ /r$ (so for
example, $]a,b[$ is empty if $b=a-1$). For a subset $S\subset \ZZ /r$,
let $t_{S} $ denote the product $\prod _{s\in S}t_{s}$. Applying the
lemma we get
\begin{align}
\nonumber\sum _{\nvec } [\Coh _{\nvec }^{(0,*\neq 0)} (\Cr )] \tvec ^{\nvec } &=
\sum _{\{\lambda (a,b) \}} (t_{1}\dotsb t_{r})^{|\lambda (a,b)|} \prod _{j\not \in [a,b]}t_{j}^{-l (\lambda (a,b))}\\
\nonumber&=\prod _{(a,b)} \left(\sum _{\lambda (a,b)} t^{|\lambda  (a,b)|}\cdot t_{]a,b[}^{-l (\lambda (a,b))} \right) \\
\nonumber&=\prod _{(a,b)} \prod _{m=1}^{\infty }\left(1-t_{]a,b[}^{-1}t^{m} \right)^{-1}\\
\nonumber&=\prod _{(a,b)} \prod _{m=1}^{\infty } \left(1-t_{[a,b]}t^{m-1} \right)^{-1}.
\end{align}
The equality from the second to the third line in the above follows
from the well-known formula
\[
\sum _{\lambda }u^{|\lambda |}v^{l (\lambda )} = \prod _{m=1}^{\infty
} (1-vu^{m})^{-1}.
\]
Combining the above with equation~\eqref{eqn: formula for stacky
Coh(0,*) in terms of Coh(0,*neq 0) using power struct} we get
\begin{align}\label{eqn: generating series for stacky part of C/ZrxC}
\nonumber \sum _{\nvec } [\Coh ^{(0,*)}_{\nvec } (\Cr )] \tvec ^{\nvec
}
&=\left(\prod _{(a,b)}\prod _{m=1}^{\infty } (1-t_{[a,b]}t^{m-1})^{-1} \right)^{\frac{\LL }{\LL -1}}\\
\nonumber &=\prod _{k=0}^{\infty }\prod _{(a,b)}\prod _{m=1}^{\infty } (1-t_{[a,b]}t^{m-1})^{-\LL ^{-k}}\\
&=\prod _{m=1}^{\infty }\prod _{k=0}^{\infty }\prod _{(a,b)} (1-\LL
^{-k}t_{[a,b]}t^{m-1})^{-1}.
\end{align}
Finally, a coherent sheaf on $\Cr $ is the direct sum of a sheaf
supported on $\CC ^{*}/ (\ZZ /r)\times \CC $ with a sheaf supported on
$B\ZZ /r\times \CC $, and the numerical invariants $\nvec $ are
additive. Consequently we get the equation
\[
\sum _{\nvec} [\Coh _{\nvec } (\Cr )] \tvec ^{\nvec } = \left(\sum _{\nvec} [\Coh^{(*\neq 0,*)} _{\nvec } (\Cr )] \tvec ^{\nvec }\right)\left(\sum _{\nvec} [\Coh^{(0,*)} _{\nvec } (\Cr )] \tvec ^{\nvec } \right)
\]
Substituting equations \eqref{eqn: series for C(n1,..,nr) is the
series for the stack Coh}, \eqref{eqn: generating series for
non-stacky part of C/ZrxC}, and \eqref{eqn: generating series for
stacky part of C/ZrxC} into the above equation yields
\[
\sum _{\nvec }\frac{[C (\nvec )]}{[G (\nvec )]} \tvec ^{\nvec } = 
\left(\prod _{m=1}^{\infty } (1-\LL t^{m})^{-1} \right)
\left(\prod _{m=1}^{\infty }\prod _{k=0}^{\infty }\prod _{(a,b)} (1-\LL
^{-k}t_{[a,b]}t^{m-1})^{-1} \right)
\]
which is easily rewritten as the equation in Theorem~\ref{thm: formula
for C (n1,...,nr)} and thus completes its proof.

\bibliography{mainbiblio}
\bibliographystyle{plain}

\end{document}